\documentclass[11pt,reqno]{amsproc}
\usepackage[margin=1in]{geometry}
\usepackage{amsmath, amsthm, amssymb}
\usepackage{times, esint,stackrel}
\usepackage{enumitem}
\usepackage{color}
\usepackage[colorlinks=true]{hyperref}
\usepackage[color=yellow]{todonotes}
\usepackage{mathtools}
\usepackage{etoolbox}
\usepackage{mathrsfs}
\usepackage{framed}

\makeatletter
\patchcmd\@thm
  {\let\thm@indent\indent}{\let\thm@indent\noindent}%
  {}{}
\makeatother

\usepackage{etoolbox}
\expandafter\patchcmd\csname\string\proof\endcsname
  {\normalparindent}{0pt }{}{}

\newcommand{\eee}{equation}

\newcommand{\pa}{\partial}

\newcommand{\be}{\begin{\eee}}
\newcommand{\ee}{\end{\eee}}

\newcommand{\eea}{\end{\eea}}

\newtheorem{prop}{Proposition}
\newtheorem{lemma}{Lemma}

\theoremstyle{definition}
\newtheorem{rem}{Remark}

\renewcommand{\div}{{\mbox{div}\,}}

\newcommand{\ve}{{\varepsilon}}

\newcommand{\rmd}{{\rm d}}

\def\div{\text{div}}

\newcommand{\lp}{\mathsf{m}}

\DeclareMathOperator{\hess}{{\rm Hess}}

\newtheorem{manualtheoreminner}{Theorem}
\newenvironment{manualtheorem}[1]{%
  \renewcommand\themanualtheoreminner{#1}%
  \manualtheoreminner
}{\endmanualtheoreminner}

\title{ \vspace{-15mm} Islands in stable fluid equilibria}

\author{Theodore D. Drivas}
\address{Department of Mathematics, Stony Brook University,
Stony Brook, NY, 11794}
\email{tdrivas@math.stonybrook.edu}
\author{Daniel Ginsberg}
\address{Department of Mathematics, Princeton University, Princeton, NJ 08544}
\email{ dg42@princeton.edu}

\date{today}
\begin{document}

\vspace{-4mm}

\begin{center}
\textit{Dedicated to Prof. Peter Constantin on the occasion of his 70th birthday.}
\end{center}
\vspace{-0mm}

\begin{abstract}
We prove that stable fluid equilibria with trivial homology on curved, reflection-symmetric periodic channels must posses ``islands", or cat's eye vortices. In this way, arbitrarily small disturbances of a flat boundary cause a change of streamline topology of stable steady states.
\end{abstract}
\vspace*{-26mm}
\maketitle

\vspace{2mm}

Given a smooth periodic function $h:\mathbb{T}\to \mathbb{R}$ such that $|h|< 1$, consider an annular domain $D_h$ with reflection symmetry across the centerline,
\be\label{dom}
D_h = \{ (x,y) \ : \ x\in \mathbb{T}, \ -1- h(x) \leq y \leq 1+h(x)\}.
\ee
We are interested in the structure of ideal fluid equilibria, i.e. steady solutions of the Euler equation, on $D_h$
\begin{align}\label{ee1}
u \cdot \nabla u &= -\nabla p \qquad \text{in } \quad D_h,\\ \label{ee2}
 \nabla \cdot u &= 0  \qquad \quad \ \ \text{in } \quad D_h,\\ \label{ee3}
 u\cdot \hat{n} &= 0  \qquad \quad \ \ \text{on } \quad \partial D_h.
\end{align}
Since $u$ is divergence-free and tangent to the boundaries, there exists a streamfunction $\psi:D_h\to \mathbb{R}$  so that  $u=\nabla^\perp \psi = (-\partial_y \psi, \partial_x \psi)$.
Of particular interest are Arnold stable solutions \cite{AK}, which satisfy
\begin{align}\label{arnoldstab}
\omega= F(\psi) \quad \text{for} \quad -\lambda_1 < F' < 0 \ \  \text{or} \ \ F'>0
\end{align}
for Lipschitz $F:\mathbb{R}\to \mathbb{R}$, where $\lambda_1= \lambda_1(D_h)$ is the smallest eigenvalue of the Dirichlet Laplacian $-\Delta$ on $D_h$.
It turns out that all Arnold stable equilibria $u=(u_1,u_2)$ on $D_h$ with trivial homology (trivial projection onto harmonic vector fields) must conform to the symmetry of the domain (see Lemma \ref{symlem}), i.e.
\begin{align} \label{sym}
 u_1(x,y) = - u_1(x,-y), \quad u_2(x,y) = u_2(x,-y).
\end{align}
    \begin{figure}[h!]\label{bstruct}
      \includegraphics[width=.42 \linewidth]{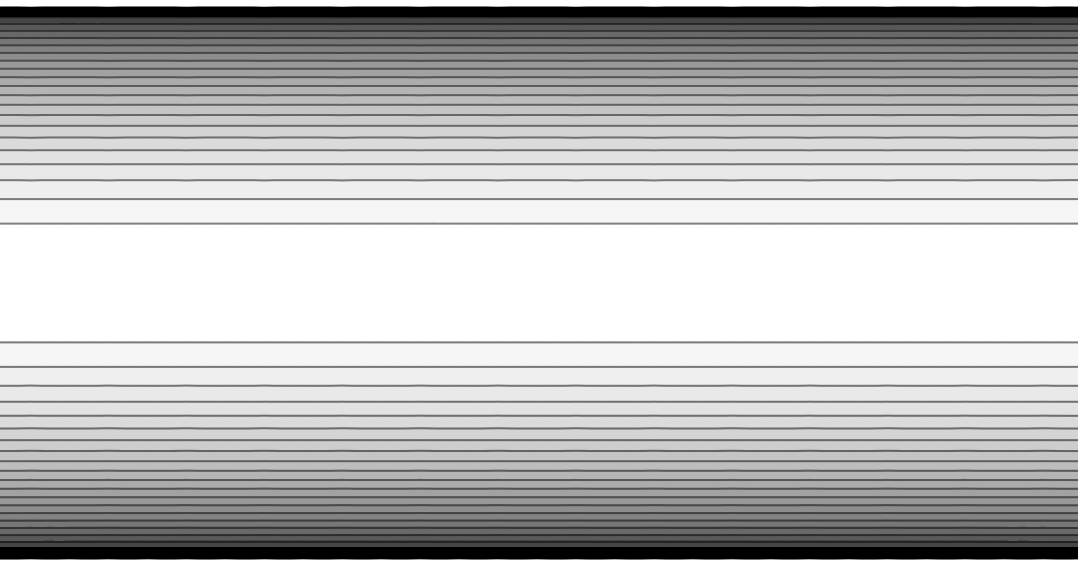}\qquad
            \includegraphics[width=.42 \linewidth]{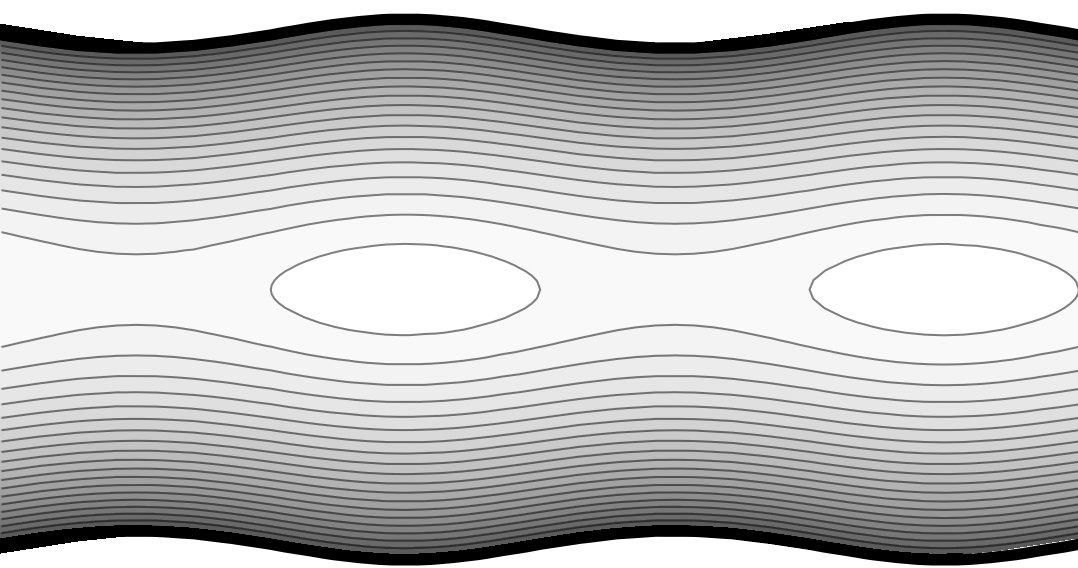}
            \caption{Stable equilibria (constant vorticity) on straight and curved channel.}
    \end{figure}

When $h'=0$,  all stable equilibria are shears, e.g. their streamlines are all straight (topologically, non-contractable loops), see \cite[Proposition 1.1]{CDG21}.
 A well-known example
is Couette flow $u(x,y) = (y,0)$. Note, since all harmonic vector fields are $\propto e_1$,  that Couette has  trivial homology as $\int_{D_0} u_1\rmd x=0$.
On the other hand, we prove here that if $h'\neq 0$, \textit{all} stable solutions with trivial homology   must also possess  streamlines that are contractible loops, e.g. they must have ``islands" or cat's eye vortices.

\begin{manualtheorem}{1}\label{thm}
Let $u\in C^1(D_h)$ be a steady Euler solution with trivial homology and Lipschitz vorticity profile with $F'>-\lambda_1$. Then either $h'=0$ or $u$ possesses a contractable streamline, i.e. it has an island.
\end{manualtheorem}
\begin{rem}
Since they are stable with trivial homology, the velocities discussed  in Theorem \ref{thm} will always have stagnation points, or even lines, with Couette flow being a prime example.
The conclusions of the theorem apply more generally to any reflection-symmetric
solution, regardless of the (Lipschitz) profile $F$ as the assumption that $F' > -\lambda_1$ is used only to establish
 reflection symmetry via Lemma \ref{symlem}.
\end{rem}

These flows possessing contractable streamlines have what are called cat's eyes, of which the Kelvin-Stuart vortex is an explicit example \cite{K,S}. These structures appear also in the plasma literature where $u$ represents the magnetic field and contractable streamlines are termed  ``magnetic islands" \cite{HK,ZHQB}. The Kelvin-Stuart vortex has been shown to be nonlinearly stable \cite{HMR,Lin} and are thus dynamically persistent. Our result shows that, in a certain sense, they are a ubiquitous feature of stable fluid equilibria.\footnote{We note that Arnol'd's condition \eqref{arnoldstab} is sufficient for stability but not necessary, at least on the \textit{straight channel} $D_0$. For instance, Poiseuille flow $u=(y^2-\frac{1}{3},0)$ is nonlinearly stable, has trivial homology, but lacks reflection symmetry. Here, $\psi= G(\omega)$ with $G(\omega)= \frac{\omega}{24} (4-\omega^2)$. Thus $G '(\omega)= \frac{1}{8}(\frac{4}{3}-\omega^2)$ which vanishes $\omega|_{y= \pm \frac{1}{\sqrt{3}}} = \frac{2}{\sqrt{3}}$, e.g. along the stagnant streamlines $\{\psi = \pm -\frac{2}{9\sqrt{3}} \}$, violating both of Arnol'd's conditions.  As such, $F:= G^{-1}$ is only H\"{o}lder continuous near the critical levels and Lemma \ref{symlem} does not apply.  On the other hand, by adding a sufficiently large mean flow, $\mathsf{U} e_1$ with $\mathsf{U}>\frac{1}{3}$, one destroys the stagnation set and $u=(y^2-\frac{1}{3}+\mathsf{U} ,0)$ becomes stable in the sense of Arnol'd \eqref{arnoldstab}.  However, this changes the homology of the flow $\mathbf{P}_{\mathcal{H}_1} u= \mathsf{U}e_1$ and, again, Lemma \ref{symlem} does not apply.   In a sense, the reason why stable flows \textit{not} satisfying  Arnol'd's condition can exist at all on the channel  is that the harmonic vector fields $\propto e_1$ happen to be simultaneously Killing vector fields for the Euclidean metric on $D_0$ and tangent to the boundary.  This leads to a ``Galilean symmetry group" mapping all solutions with a given mean flow to solutions with any other and hence stability is transferred despite the breaking of  condition \eqref{arnoldstab}, see \cite[Section 2.2.]{DE22}. }

\begin{rem}[Islands are destroyed by current]
The assumption in Theorem \ref{thm} on trivial homology can, in general, not be dropped. Indeed, there exist smooth Arnol'd stable steady states in any $D_h$ with $h'\neq 0$ sufficiently small having non-trivial homology such that all  are non-contractable loops, i.e. there are no islands.  This follows directly from the results of \cite{CDG21}. Specifically,  consider the velocity $u=\nabla^\perp \psi$ with  $\psi(x,y) =   -\tfrac{1}{2}y^2+1.01y$ on the channel $D_0=\mathbb{T}\times [-1,1]$, which has the property that $|\nabla \psi|>0.01$ on $D_0$.  Then by  \cite[Theorem 1.2]{CDG21}, for small but arbitrary deformations $h$ of the boundary there exists a diffeomorphism $\varphi$ taking the original domain to the perturbed domain such that $\psi\circ \varphi$ is the streamfunction of a stable stationary solution on the new domain. The resulting solution has the same topology as $\psi$.
\end{rem}

    \begin{figure}[h!]\label{bstruct2}
          \includegraphics[width=.45 \linewidth]{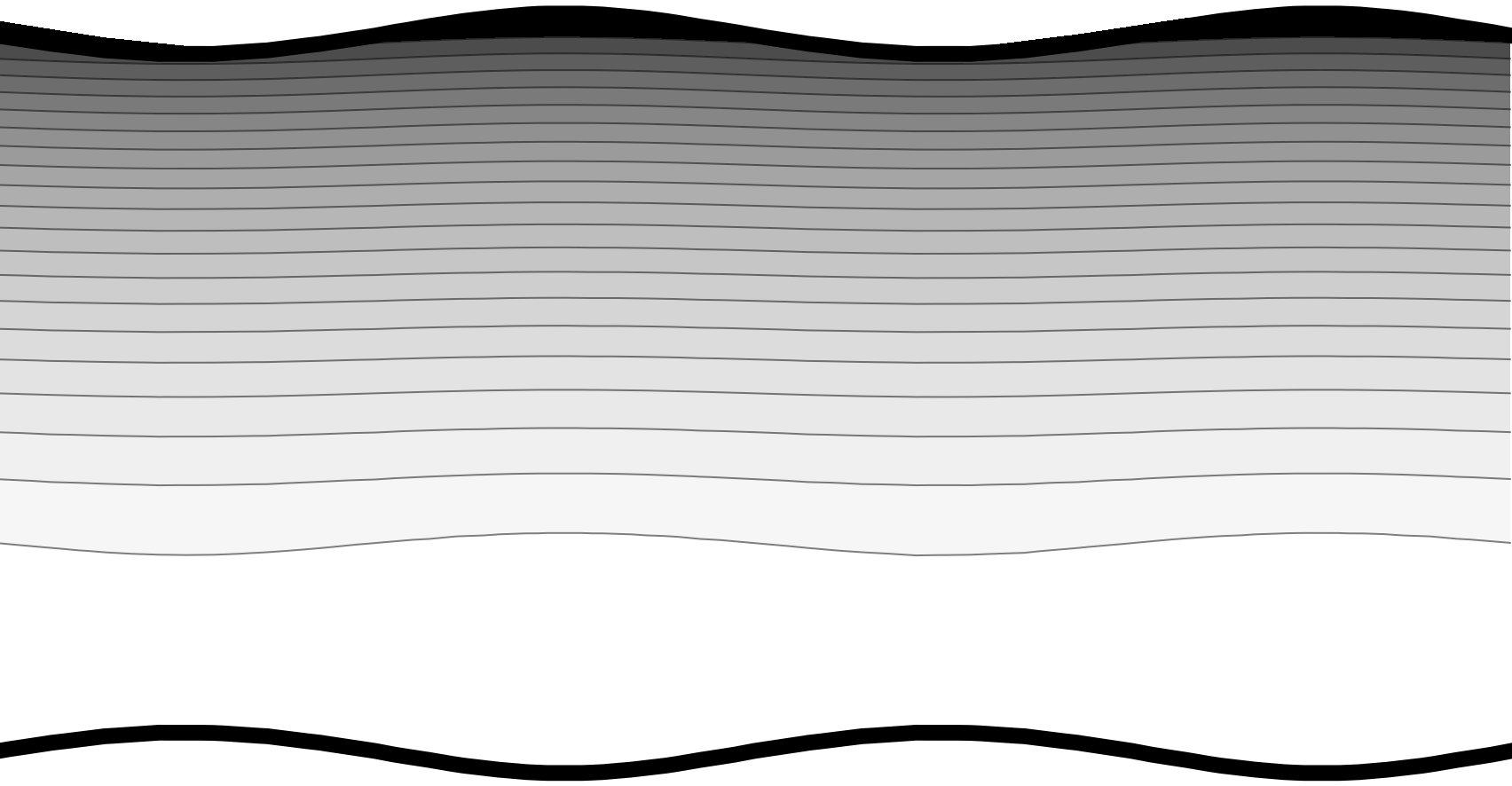}
            \caption{Steady states with no islands with non-trivial homology.}
    \end{figure}

We now prove Theorem \ref{thm}. We begin by establishing symmetry under the stated hypotheses.
\begin{lemma}\label{symlem}
Suppose that $u=\nabla^\perp \psi$ has trivial homology and
 \begin{align}
  \Delta \psi &= F(\psi)  \qquad \text{in } \quad D_h,\\
	 \psi &= {\rm const}  \qquad \text{on } \quad \partial D_h,
 \end{align}
 for  $F' > -\lambda_1$. Then $u$ has reflection symmetry, i.e. it satisfies \eqref{sym}.
\end{lemma}
\begin{proof}
The space of harmonic vector fields tangent to the boundary $\mathcal{H}_1(D_h)$ is
\be
\mathcal{H}_1(D_h)= \{ \nabla^\perp q \ : \ f\in C^\infty(\overline{D}_h), \ \Delta q = 0 \ \text{on} \ D_h, \ q=c_j \ \text{on} \ \Gamma_j, \ j=1,2\}
\ee
where $c_j$ are arbitrary constants and $ \Gamma_j$ are the two connected components of $\partial D_h$. This space is  one-dimensional as  the only free parameter is the difference in boundary values $c_2-c_2$, since the streamfunction $q$ is defined up to an additive constant. The $L^2$ orthogonal projection $\mathbf{P}_{\mathcal{H}_1}$ onto $\mathcal{H}_1(D_h)$ is given by
\be\label{homcalc}
(\mathbf{P}_{\mathcal{H}_1} u ,q)_{L^2}= \int_{D_h} u \cdot \nabla^\perp q \rmd x  =  \psi|_{\Gamma_2}\int_{\Gamma_2} \partial_n q  -\psi|_{\Gamma_1} \int_{\Gamma_1} \partial_n q = \big(\psi|_{\Gamma_2}-\psi|_{\Gamma_1}\big)\int_{\Gamma_1} \partial_n q
\ee
where $q$ generates $\mathcal{H}_1(D_h)$ and has unit $L^2$ norm, and where we used that $\int_{\Gamma_2} \partial_n q- \int_{\Gamma_1} \partial_n q = \int_{D_h} \Delta q = 0$. As a result, $\mathbf{P}_{\mathcal{H}_1} u=0$ implies that $\psi$ takes the same value on both connected components of $\partial D_h$.

	We now show that any such streamfunction has an even symmetry about the line $\{y=0\}$.  To this end, let $\psi_{\mathsf{R}}(x,y) := \psi(x,-y)$.
	The function $\phi = \psi -\psi_{\mathsf{R}}$ satisfies
	\begin{equation}
	 \Delta \phi = G(\psi)\phi,
	 \qquad G(\psi) = \begin{cases} \frac{F(\psi) - F(\psi_{\mathsf{R}})}{\psi-\psi_{\mathsf{R}}}, \qquad &\psi\not=\psi_{\mathsf{R}},\\
	 F'(\psi), \qquad &\psi =\psi_{\mathsf{R}} \end{cases}.
	 \label{}
	\end{equation}
	Since $\phi = 0$ on $\pa D_h$, integrating by parts yields
$
	 \int_{D_h} |\nabla \phi|^2 = -\int_{D_h} G(\psi) \phi^2.
$
	On the other hand, by Poincar\'{e}'s inequality,  $
		\lambda_1 \int_{D_h} \phi^2 \leq \int_{D_h} |\nabla \phi|^2$
	and so combining the above we find
	\begin{equation}
	 \int_{D_h} \left(\lambda_1 + G(\psi)\right)\phi^2 \leq 0.
	 \label{pineq}
	\end{equation}
	By the mean value theorem, we have $G(\psi) = F'(\xi)$ for some $\xi = \xi(x,y)$
	lying between $\psi(x,y)$ and $\psi_{\mathsf{R}}(x,y)$. Since $F'>-\lambda_1$,
		the inequality \eqref{pineq} forces $\phi = 0$.
\end{proof}
The reflection symmetry, together with the assumed $C^1$ smoothness,  forces $u_1|_{y=0}=0$. We now show
\begin{prop}
Under the conditions of Theorem \ref{thm}, $u_2|_{y=0}\neq 0$ on $\{y=0\}$ unless $h'=0$ or $u=0$.
\end{prop}

\begin{proof}
We shall show that, under our hypotheses, demanding $u_2|_{y=0}= 0$ forces $\psi$ to solve an overdetermined elliptic problem and this results in $\psi$ being constant unless $h'=0$.  This is similar to the rigidity encountered in stationary free boundary fluid problems \cite{HN,CDG22}. In this case, the rigidity will follow from a unique continuation argument via the following Carleman estimate.
\begin{lemma}
	\label{generalcarlemanlem}
 Let  $D$ be a bounded domain and $w \in H^2({D})$ satisfy $w, \nabla w = 0$ on $\pa {D}$.
 Let $\varphi_0 \in C^4({D})$ and suppose that
 $\varphi_0 \geq 0$ and
 $|\pa_y \varphi_0| > 0$ in ${D}$.
 Then, for $C, \lp_0, \lambda_0 > 0$ depending only on ${D}$
  and $\|\varphi_0\|_{C^4({D})}$, $(\inf_D |\nabla \varphi_0|)^{-1}$,
 if $|\lp|> \lp_0$ and $\lambda > \lambda_0$, we have
 \begin{equation}
  \int_{{D}} |e^{\lp \varphi} \Delta e^{-\lp \varphi} w|^2\,\rmd x
	\geq C \lp^2 \int_{{D}} |w|^2\, \rmd x ,
	\qquad \text{ where } \varphi := e^{\lambda \varphi_0}.
  \label{generalcarleman}
 \end{equation}
 \end{lemma}
 \begin{proof}
 The proof is motivated by arguments that can be found, e.g. in \cite{Horm}.
 We first claim that, regardless of the specific form of $\varphi$, if
 $|\pa_y \varphi| >0$ for any $\lp$ we have the following inequality,
 \begin{equation}
  \int_{{D}} |e^{\lp \varphi} \Delta e^{-\lp \varphi} w|^2\,\rmd x
  \geq
	  \int_{{D}} |Aw |^2
	\, \rmd x
  + \int_{{D}} 	K[w,\varphi]\, \rmd x,
  \label{carlemanvarphi}
 \end{equation}
 where $A := \Delta + \lp^2 |\nabla \varphi|^2 $ and
 where,
  writing $\hess \varphi = \nabla \otimes \nabla \varphi$,
 \begin{equation}
 	K[w,\varphi]:=
  4\lp (\hess\varphi)(\nabla w, \nabla w) +
	 4 \lp^3 (\hess\varphi)(\nabla \varphi,
	\nabla \varphi) |w|^2 - \lp \Delta^2 \varphi | w|^2.
  \label{Kformula}
 \end{equation}
\textbf{Proof of Inequality \eqref{carlemanvarphi}.}
Write  $  e^{\lp \varphi} \Delta e^{-\lp \varphi} w = Aw + Bw$
 where $B = -\lp \left(2\nabla \varphi \cdot \nabla + \Delta \varphi\right)$.
Then
 \begin{equation}
  \int_{{D}} |Aw + Bw|^2\, \rmd x
	=\int_{{D}}\left( |Aw|^2 + |Bw|^2 + 2 Aw Bw\right)\, \rmd x
	\geq \int_{D} \left(|Aw|^2 + 2 Aw Bw\right)\, \rmd x.
  \label{}
 \end{equation}
 To deal with the last term, we claim that
 \begin{equation}
  2Aw Bw = -2\lp \left( \Delta w+ \lp^2 |\nabla \varphi|^2w\right)
	\left( 2\nabla\varphi\cdot \nabla w + (\Delta \varphi)w\right)
	=   \div T + K,
  \label{divident}
 \end{equation}
 where $K$ is as in \eqref{Kformula} and where
 \begin{align}\label{Tformula}
  T &
	=-2\lp \left(\nabla w \nabla \varphi \cdot \nabla w -\tfrac{1}{2} \nabla \varphi
	|\nabla w|^2 + \nabla w w \Delta \varphi - \tfrac{1}{2} |w|^2 \nabla
	\Delta \varphi + \lp^2|\nabla \varphi|^2 \nabla \varphi |w|^2 \right)
\end{align}
has the property that $\int_{\pa {D}} T\cdot \widehat{n} = 0$ under
our assumptions on $w$. The bound \eqref{carlemanvarphi} then follows immediately
from the identity \eqref{divident} and \eqref{Tformula}.
To prove this identity, we first write
 \begin{align}\nonumber
  2\Delta w \nabla \varphi \cdot \nabla w
	&= \div\left( 2\nabla w \nabla\varphi \cdot \nabla w -  \nabla \varphi |\nabla w|^2\right)
	+  \Delta \varphi |\nabla w|^2
	- 2(\hess\varphi)(\nabla w, \nabla w).
  \label{}
 \end{align}
 We also have
 \begin{multline}
  \Delta \varphi \Delta w  w  = \Delta \varphi\, \div (\nabla w w)
	- \Delta \varphi |\nabla w|^2
	= \div \left( \Delta \varphi \nabla w w  - \tfrac{1}{2} |w|^2 \nabla \Delta
	\varphi\right) + \tfrac{1}{2}\Delta^2 \varphi |w|^2 - \Delta \varphi|\nabla w|^2
\nonumber
 \end{multline}
 as well as
\begin{equation}
  \lp^2|\nabla \varphi|^2 w \left(2 \nabla \varphi\cdot \nabla w + \Delta \varphi w\right)
 = \div \left( \lp^2 |\nabla \varphi|^2 \nabla \varphi |w|^2\right)
 - 2\lp^2 (\hess \varphi)(\nabla \varphi, \nabla \varphi) |w|^2,
 \label{}
\end{equation}
and adding up the previous three lines and multiplying
by $-2\lp$ we arrive at \eqref{divident}.
\vspace{2mm}

\noindent \textbf{Proof of Inequality \eqref{generalcarleman}.} With $\varphi = e^{\lambda \varphi_0}$, we compute
\begin{align}
	\label{hessformula}
 \nabla \varphi &= \lambda \nabla \varphi_0 \varphi,\qquad \hess \varphi = \lambda\left( \hess \varphi_0  + \lambda \nabla\varphi_0\otimes\nabla \varphi_0\right)  \varphi,
 \qquad
 \Delta \varphi = \lambda \left( \Delta \varphi_0 + \lambda |\nabla \varphi_0|^2\right) \varphi
 \end{align}
 and so $K$ takes the form
 \begin{multline}
  K =4\lp \lambda^2 |\nabla \varphi_0\cdot \nabla w|^2 \varphi^2
	+4 \lp^3 \lambda^4 |\nabla \varphi_0|^4 |w|^2 \varphi^3\\
	+4 \lp \lambda (\hess \varphi_0)(\nabla w, \nabla w)\varphi
	+ 4 \lp^3 \lambda^3 (\hess \varphi_0)(\nabla \varphi_0, \nabla \varphi_0)
	|w|^2 \varphi^3- \lp \Delta^2 \varphi |w|^2.
  \label{}
 \end{multline}
 Bounding $|\Delta^2 \varphi| \leq C \lambda^4 \varphi$ where $C$
 depends on $\|\varphi_0\|_{C^4}$ and noting the positivity
 of the terms on the first line,
since $|\nabla \varphi_0|$ is bounded below in $D$, there are $\lambda_0 > 0$
and $\lp_0 > 0$
depending on $ (\inf |\nabla \varphi_0|)^{-1} > 0$ and
$\|\varphi_0\|_{C^2(D)}$ so that if $\lambda > \lambda_0$
and $\lp > \lp_0$ we have
\begin{equation}
 4 \lp^3 \lambda^4 |\nabla \varphi_0|^4 |w|^2\varphi^3 + 4 \lp^3 \lambda^3 (\hess \varphi_0)(\nabla \varphi_0, \nabla \varphi_0)
 |w|^2\varphi^3 - \lp \Delta^2 \varphi |w|^2
 \geq C \lp^3 \lambda^4 |w|^2\varphi^3,
 \label{}
\end{equation}
and for such $\lambda$ we find
\be
 	K[w,\varphi]\geq
	C \left( \lp \lambda^2 |\nabla \varphi_0 \cdot \nabla w|^2 \varphi^2 + \lp^3 \lambda^4 |w|^2 \varphi^3\right)
	- C' \lp\lambda |\nabla w|^2 \varphi,
	\label{almost0}
\ee
and returning to \eqref{carlemanvarphi}, we have the inequality
\begin{equation}
 \int_{{D}} |e^{\lp \varphi} \Delta e^{-\lp \varphi} w|^2\,\rmd x
 \geq
	C \int_{{D}} \left(|Aw |^2 + \lp \lambda^2 |\nabla \varphi_0\cdot \nabla w|^2
	\varphi^2 + \lp^3 \lambda^4 |w|^2 \varphi^3\right)\, \rmd x
	- C'\lp\lambda  \int_{D}|\nabla w|^2 \varphi\, \rmd x.
 \label{almost}
\end{equation}
To control the last term here, we start by getting a lower bound
for $\|A w\|_{L^2}$.
For this, we write
\begin{align}\nonumber
 \int_{D} |\nabla w|^2 \,\varphi \rmd x
 &= -\int_{D}  Aww\, \varphi \rmd x
 + \int_{D}\Big(\lp^2|\nabla \varphi|^2 |w|^2\, \varphi
 - \nabla \varphi \cdot \nabla w w\Big) \rmd x\\ \nonumber
 &= -\int_{D}   Aww\, \varphi\rmd x +\int_{D}
 \lp^2\lambda^2 |\nabla \varphi_0|^2
  |w|^2 \varphi^3\rmd x
- \int_{D} \lambda  (\nabla \varphi_0 \cdot \nabla w) w\, \varphi \rmd x\\ \nonumber
 &\leq \frac{1}{2 \delta}
 \| A w\|_{L^2}^2 +C
 \Big( \delta + \lp^2\lambda^2 \Big)
 \| w \varphi^{3/2}\|_{L^2}^2
 + \lambda \|\nabla \varphi_0\cdot \nabla w  \|_{L^2}\| w \varphi\|_{L^2}
 \label{}
\end{align}
for any $\delta > 0$,
where $C=C(\|\varphi_0\|_{C^1(D)})$, $L^2 = L^2(D)$. Here, we used
$\varphi,\lp, \lambda \geq 1$. In particular,
\begin{align}\nonumber
 \frac{1}{2} \|A w\|_{L^2}^2
 &\geq \delta\| (\nabla w) \varphi^{1/2} \|_{L^2}^2-
  C
	\Big( \delta^2 + \delta\lp^2\lambda^2 \Big)
	 \| w \varphi^{3/2}\|_{L^2}^2
 - C\lambda \delta\|\nabla \varphi_0\cdot \nabla w\|_{L^2}\| w \varphi\|_{L^2}.
 \label{}
\end{align}
Returning to \eqref{almost}, we have shown that for $\lambda, \lp$ sufficiently
large and any $\delta > 0$,
\begin{align}\nonumber
   \|e^{\lp \varphi} \Delta e^{-\lp \varphi} w\|_{L^2}^2
	&\geq
	C\delta\|(\nabla w)\varphi^{1/2}\|_{L^2}^2+ C \lp^3 \lambda^4 \|w \varphi^{3/2}\|_{L^2}^2
	+ C \lp^2 \lambda^2 \|(\nabla \varphi_0\cdot \nabla w) \varphi\|_{L^2}^2
	\\\nonumber
	&\quad-
	C \lp\lambda \| (\nabla w)\varphi^{1/2}\|_{L^2}^2- C
	\Big( \delta^2 + \delta\lp^2\lambda^2 \Big)
	\| w \varphi^{3/2}\|_{L^2(D_h)}^2
	-
	C\lambda \delta\|\varphi_0\cdot \nabla w \|_{L^2}\| w \varphi\|_{L^2},
 \label{}
\end{align}
We now want to choose $\delta > 0$ so that we can absorb the terms on the
second line into the terms on the first line. For this we take
$\delta = \delta' \lp \lambda^2$ for small $\delta' > 0$, and the above becomes
\begin{align}\nonumber
 \|e^{\lp \varphi} \Delta e^{-\lp \varphi} w\|_{L^2}^2
&\geq
 C\lp\lambda^2\left( \delta'  \lambda - C' \right)\|(\nabla w)\varphi^{1/2}\|_{L^2}^2+C \lp^2 \lambda^4 \left( (1-C'\delta')\lp - C'(\delta')^2\right)\| w \varphi^{3/2}\|_{L^2(D_h)}^2\\
 &\quad
 + C \lp^2 \lambda^2 \|(\nabla \varphi_0\cdot \nabla w) \varphi\|_{L^2}^2
 -
 C\delta'\lp \lambda^3 \|\nabla\varphi_0\cdot \nabla w \|_{L^2}\| w \varphi\|_{L^2}
\end{align}
where all the above constants depend only on
 $\|\varphi_0\|_{C^4(D)}$. Taking $\delta'$ small enough that the coefficient
of the second term is bounded below by $C\lp^3 \lambda^3/2$ whenever
$\lp \geq 1$, and then taking $\lambda$ large enough that the
coefficient of the first term is bounded below by $C\lp \lambda^3/2$,
and finally taking $\lp$ large enough that for this choice
of $\lambda, \delta'$, we can absorb it into the first
term on the last line and the second term on the first line, we have
arrived at the estimate
\begin{equation}
 \|e^{\lp \varphi} \Delta e^{-\lp \varphi} w\|_{L^2}^2
\geq
C \lp \lambda^3 \| (\nabla w)\varphi^{1/2}\|_{L^2}^2
+ C \lp^3 \lambda^4\| w \varphi^{3/2}\|_{L^2}^2
+ C \lp^2 \lambda^2 \|(\nabla \varphi_0\cdot \nabla w)\varphi\|_{L^2}^2,
 \label{}
\end{equation}
for large enough $\lambda, \lp$,
which is more than sufficient for
the estimate \eqref{generalcarleman} since $\varphi \geq 1$.
 \end{proof}
 We use this to prove
\begin{lemma}
	Let $h \in C^2(\mathbb{R})$ and
	set $D_h^{\mathsf{up}} = \{(x, y) : 0 \leq y\leq 1+ h(x), x \in \mathbb{T}\}$.
	Let $F$ be a Lipschitz function and suppose that $\psi \in H^2(\Omega)$ satisfies
 \begin{alignat}{2}  \label{elipeqn}
  \Delta \psi &= F(\psi), &&\qquad \text{ in } D_h^{\mathsf{up}},\\
	\pa_x\psi &= \pa_y \psi= 0&&\qquad   \text{ on } \{y=0\}.
\end{alignat}
Then $\pa_x \psi = 0$.  Moreover, if $\psi$ is constant at $\{y= 1+ h(x)\}$ then either $h' =0$ or $\pa_y \psi = 0$ as well.
\end{lemma}
\begin{proof}
	We first show that $\pa_x\psi = 0$.
Since $F$ is a Lipschitz function, it is differentiable almost everywhere
and so by elliptic regularity, $\psi \in H^3(D_h^{\mathsf{up}})$. Now we note that $\pa_x \psi$
	satisfies
\begin{alignat}{2}
 \Delta \pa_x \psi&= F'(\psi)\pa_x\psi, &&\qquad \text{ in } D_h^{\mathsf{up}},\\
 \pa_x\psi &= \pa_y \pa_x\psi = 0&&\qquad   \text{ on } \{y=0\}.
 \label{}
\end{alignat}
Here, the fact that $\pa_y \pa_x\psi|_{y = 0}$ is well-defined follows
from the trace theorem and the fact that  $\psi \in H^3$.
Now define a smooth increasing cutoff function $\chi_0(z)$ so that $\chi_0(z) = 1$
when $z \geq 1$ and $\chi_0(z) = 0$ when $z \leq 1/2$.
Let $\varphi_0(x,y) = 1+h(x) - y$ and set $\varphi(x, y) = e^{\lambda \varphi_0(x,y)}$
where $\lambda > \lambda_0$ with $\lambda_0$ as in Lemma \ref{generalcarlemanlem}.
Then $D_h^{\mathsf{up}} = \{(x,y) \in \mathbb{T}\times\mathbb{R}: \varphi(x, y) \geq 1, y \geq 0\}$ and $\varphi > 1$ in the interior.
If we define
$\chi_c(x, y) = \chi_0((\varphi(x,y)-1)/c)$
and $\Omega_c = \chi_c^{-1}(1)$, then $1 + c \leq \varphi(x,y)$ on $\Omega_c$.
We also have that $\varphi(x, y) \leq 1+c$ on the support of $\nabla \chi_c$.
Moreover for $c > 0$, $\chi_c$ is zero near the boundary $\{y =1+ h(x)\}$.

Under our hypotheses on $h$, the assumptions of Lemma \ref{generalcarlemanlem}
hold, and so taking $\lp$ sufficiently large and
applying the Carleman estimate \eqref{generalcarleman}
to $w = e^{\lp \varphi} \chi_c \pa_x\psi$ with $c > 0$, we find
\begin{align}\nonumber
  \|e^{\lp \varphi} \chi_c \pa_x\psi\|_{L^2(D_h^{\mathsf{up}})}
 &\leq \frac{C}{\lp} \|e^{\lp \varphi} \Delta (\chi_c \pa_x\psi)\|_{L^2(D_h^{\mathsf{up}})}\\ \nonumber
 &\leq \frac{C}{\lp} \| e^{\lp \varphi} [\Delta, \chi_c] \pa_x\psi\|_{L^2(D_h^{\mathsf{up}})}
 + \frac{C}{\lp} \| e^{\lp \varphi} F'(\psi) \pa_x\psi\|_{L^2(D_h^{\mathsf{up}})}\\
 &\leq \frac{C}{\lp} e^{\lp (1+c)} \| \pa_x\psi\|_{H^1(D_h^{\mathsf{up}})}
 + \frac{C}{\lp} \|F'\|_{L^\infty(\textrm{rang} (\psi))}\| e^{\lp \varphi}\chi_c  \pa_x\psi\|_{L^2(D_h^{\mathsf{up}})},
 \label{}
\end{align}
where $C$ depends continuously on $1/c$.
Here we used the equation for $\pa_x\psi$ and that $\varphi(x,y)\leq c+1$
on the support of the commutator $[\Delta, \chi_c]$. We also used that $c > 0$ to justify
the vanishing of $w$ at the top boundary and
that $\pa_x\psi, \nabla \pa_x\psi = 0$ when $y = 0$ to justify the vanishing at the bottom
boundary. If we now take $\lp$ so large that $\frac{C}{\lp} \|F'\|_{L^\infty(\textrm{rang} \psi)}
\leq \frac12$, we can
absorb the second term on the right-hand side into the left, giving
\begin{equation}
	e^{\lp (1+c)} \| \pa_x\psi\|_{L^2(\Omega_c)}
  \leq\|e^{\lp \varphi} \chi_c \pa_x\psi\|_{L^2(D_h^{\mathsf{up}})}
 \leq \frac{C}{\lp} e^{\lp (1+c)} \| \pa_x\psi\|_{H^1(D_h^{\mathsf{up}})}, \qquad c > 0.
 \label{}
\end{equation}
Dividing both sides by $e^{\lp(1+c)}$ and taking $\lp \to \infty$ we find that
for each $c > 0$, $\pa_x\psi $ vanishes in the region $\Omega_c$.
Since $\Omega_c \to D_h^{\mathsf{up}}$ as $c \to 0$, taking
$c \to 0$ we find that $\pa_x\psi$ vanishes everywhere in $D_h^{\mathsf{up}}$.

This has the following consequence: $\psi$ is a function of one variable, $y$ and moreover
is identically constant on the set $\mathsf{S}:=\{(x,y) \ : \ x\in\mathbb{T}, \  1 +h_- \leq y\leq 1+h(x)\}$ where $h_-:=\min_{x\in \mathbb{T}}h(x)$.  Now consider the rectangular domain $\mathsf{R}:=\{(x,y) \ : \ x\in\mathbb{T}, \  0 \leq y\leq 1 +h_-\}$.  On this domain, $\pa_y\psi$ satisfies
\begin{align}
 \Delta \pa_y \psi&= F'(\psi)\pa_y\psi, \qquad   \text{ in } \mathsf{R},\\ \pa_y\psi &= \pa_y^2\psi = 0 \qquad\quad  \text{ on } \{y=1 +h_-\}.
 \label{bcR}
\end{align}
Indeed, the boundary condition \eqref{bcR} follows from the fact that $\Delta\psi= \partial_y^2 \psi = F(\psi)$, so $\partial_y^2\psi\in C(\mathbb{R})$.  Along $\{y=1+h_-\}$, the solution $\psi$ must match its constant value in $\mathsf{S}$ and thus $ \pa_y\psi = \pa_y^2\psi =0$.  Provided $h'\neq 0$ so that $\mathsf{S}\neq \emptyset$, the result that $\partial_y\psi=0$ follows by essentially the same argument as above.
\end{proof}
The Proposition follows, since if $u_2|_{\{y=0\}}=0$ then unless $h'= 0$ we must have $u=0$ everywhere.  
\end{proof}

\vspace{-2mm}
\noindent To complete the proof of Theorem \ref{thm}, we appeal to
\begin{prop}
Let $u$ be a Lipschitz divergence-free vector field on $D_h$ with reflection symmetry.  Almost every point $p$ on $\{y=0\}$ such that $|u_2(p)|\neq 0$ lies on a closed streamline.
\end{prop}
\begin{proof}
Let $\mathsf{F} := \{p\in\{y=0\} \ : \ |u_2(p)|\neq 0\}$ be the (open) set of non-stagnant points along the centerline. By incompressibility, almost every point $p\in\mathsf{F}$ is a regular point of $\psi$, namely the set $\{x \ : \ \psi(x)=\psi(p)\}$ contains no critical points.  Otherwise there would exist an interval $I\subset \mathsf{F}$ such that for almost every $q\in I$, the Omega limit set of the orbit passing through $q$ is a critical point.  Since $|u_2||_{I}> \ve_0>0$ as $u$ is continuous,  flowing for short time the interval, one can form a flowbox of positive area which would then (in infinite time) compress to a set of zero area,  a contradiction. 
For any regular point $p$, it is easy to see that the level set $\{x \ : \ \psi(x)=\psi(p)\}$ is a simple closed curve. Moreover, since the streamfunction $\psi$ has even symmetry about $\{y=0\}$, any such curve must be contractible since (a) the orbit  through $p$ must return to $\{y=0\}$ by uniqueness of integral curves and (b) once it returns, by symmetry, the integral curve is contractible. This completes the proof.
\end{proof}

\begin{rem}
In fact, if $\psi$ is a Morse function, \textit{every} point $p$ on $\{y=0\}$ such that $|u_2(p)|\neq 0$ lies on a closed streamline. This follows from the Poincar\'{e}--Hopf index theorem, the fact that the Euler characteristic of the annulus is zero, and the Poincar\'{e}-Bendixson theorem together with incompressibility which ensures all critical points are either hyperbolic or elliptic.
\end{rem}

\noindent \textbf{Acknowledgements.} We thank P. Constantin for many inspiring discussions, D. Peralta-Salas for useful comments, and A. Bhattacharjee for bringing the Hahm-Kulsrud problem to our attention.
The research of TDD was partially supported by the NSF
DMS-2106233 grant and  NSF CAREER award \#2235395.

\end{document}